\documentclass[11pt,twoside,a4paper]{article}

\author{Shai Sarussi}
\title{Totally ordered sets and the prime spectra of rings}
\date{}

\usepackage{amsmath,amsthm,amssymb}
\usepackage{verbatim}

\begin{document}

\newtheorem{thm}{Theorem}[section]
\newtheorem{cor}[thm]{Corollary}
\newtheorem{lem}[thm]{Lemma}
\newtheorem{prop}[thm]{Proposition}
\newtheorem{ax}{Axiom}

\theoremstyle{definition}
\newtheorem{defn}[thm]{Definition}

\theoremstyle{remark}
\newtheorem{rem}[thm]{Remark}
\newtheorem{ex}[thm]{Example}
\newtheorem*{notation}{Notation}

\newcommand{\qv}{{quasi-valuation\ }}


\maketitle

\begin{abstract} {Let $T$ be a totally ordered set
and let $D(T)$ denote the set of all cuts of $T$.
We prove the existence of a discrete valuation domain $O_{v}$ such that $T$ is order isomorphic to two special subsets of Spec$(O_{v})$.
We prove that if $A$ is a ring (not necessarily commutative) whose prime spectrum is totally ordered
and satisfies (K2), 
then there exists a totally ordered set
$U \subseteq \text{Spec}(A)$ such that the prime spectrum of $A$ is order isomorphic to $D(U)$.
We also present equivalent conditions for a totally ordered set to be a Dedekind totally ordered set.
At the end, we present an algebraic geometry point of view.



}\end{abstract} 
\section{Introduction and some basic terminology}

In his book Commutative Rings, Kaplansky presented two basic properties regarding the prime spectrum, Spec$(R)$,
of a commutative ring $R$. Explicitly, the first property (K1) is that every nonempty totally ordered subset of Spec$(R)$ has
an infimum and a supremum (cf. [Ka, Theorem 9]). The second property (K2) is that given two prime ideals $P_1 \subset P_2$, there exist prime ideals
$P_1 \subseteq P_3 \subset P_4 \subseteq P_2$ such that there is no prime ideal between $P_3$ and $P_4$ (cf. [Ka, Theorem 11]).
Kaplansky conjectured that these two properties characterize the prime spectrum of a commutative ring.
However, in 1969 it turned out that the conjecture was false, due to Hochster. 
In his work, Hochster characterized topological spaces appearing as the
prime spectra of commutative rings (cf. [Ho, Theorem 6 and Proposition 10]).
Speed (cf. [Sp, Corollary 1]) pointed out that Hochster's result gives the following characterization of partially ordered sets
appearing as the prime spectra of commutative rings: a partially ordered set $S$ is isomorphic to the prime spectrum of some commutative ring
if and only if $S$ is an inverse limit of finite partially ordered sets
in the category of partially ordered sets.
In 1973, Lewis (cf. [Le, Theorem 10]) proved that any finite partially ordered set is order isomorphic to
the prime spectrum of some commutative ring; Lewis' proof provides a way of constructing a ring with a desired
spectrum. Lewis also gave an explicit example showing that the properties (K1) and (K2) are not sufficient.
In 1994, Facchini (cf. [Fa, Theorem 5.3]) proved 
that a partially ordered set is order isomorphic to the prime spectrum of a generalized
Dedekind domain if and only if it is a Noetherian
tree with a smallest element.


In this paper we study some connections between totally ordered sets, Dedekind totally ordered sets,
and the prime spectra of rings having totally ordered prime spectra; in particular, valuation domains.
We prove that if $A$ is a ring such that Spec$(A)$ satisfies (K2)
(between any two comparable elements of Spec$(A)$
there exist immediate neighbors), and is totally ordered, then
Spec$(A)$ is order isomorphic to the set of all cuts of
prime ideals of $A$ which are immediate predecessors (and to the set of all cuts of
prime ideals of $A$ which are immediate successors).
We also prove that
any totally ordered set $T$
is order isomorphic to the set of all immediate predecessors (and to the set of all immediate successors) of Spec$(O_{v})$, for some discrete valuation domain $O_{v}$.
These results enable us to present several equivalent conditions for a totally ordered set
to be order isomorphic to the prime spectrum of some discrete valuation domain.


Here is a brief overview of this paper. In section 2 we study more closely the property presented in [Ka, Theorem 11].
More precisely, let $B$ be a set of sets partially
ordered by inclusion and let $X_1 \subset X_2$ be two elements of $B$. We give a necessary and sufficient condition
for the existence of immediate neighbors in $B$ between $X_1$ and $X_2$.
We conclude that if $B$ is
closed under unions and intersections of nonempty chains, then $B$ satisfies the property (K2).
We prove in Theorem
\ref{union of initial subset of IS gives every nonminimal element} that if $B$ satisfies (K2) and is closed
under unions of nonempty chains, then every non-minimal element $X$ of $B$
is the union of any maximal chain of immediate successors contained in $X$. 
In Theorem \ref{when B is a chain} we prove that if such $B$ is a chain and has a smallest element, then $B$ is isomorphic to the set of all cuts
of its immediate successors (and to the set of all cuts of
its immediate predecessors).


In section 3 we apply Theorem \ref{when B is a chain} to ring theory.
We then consider a totally ordered set $T$ and 
construct a totally ordered group $\Gamma$.
We study the ``$T$-isolated subgroups" of $\Gamma$ and prove that these are precisely the isolated
subgroups having an immediate predecessor in $\Gamma$. This enables us to prove one of the main results of this paper:
every totally ordered set is order isomorphic to the set of all prime ideals of $O_{v}$ having an immediate successor,
for some discrete valuation domain $O_{v}$. We present in Corollary \ref{equivalent conditions}
equivalent conditions for a totally ordered set to be a Dedekind totally ordered set.
Finally, we consider a ring $A$ such that Spec$(A)$ satisfies (K2) and is totally ordered,
and discuss the set of immediate predecessors and
the set of immediate successors of Spec$(A)$,
from an algebraic geometry point of view.



In this paper the symbol $\subset$ means proper inclusion,
and the symbol $\subseteq$ means inclusion or equality.

Recall that a valuation on a field $F$ is a function $v : F
\rightarrow \Gamma \cup \{  \infty \}$, where $\Gamma$ is a
totally ordered abelian group and where $v$ satisfies the following
conditions:


 (A1) $v(x) \neq \infty$ iff $  x \neq 0$, for all $x\in F$;

(A2) $v(xy) = v(x)+v(y)$ for all $x,y \in F$;

(A3) $v(x+y) \geq \min \{ v(x),v(y) \}$ for all $x,y \in F$.

\begin{defn} Let $(S, \leq)$ be a poset. A subset $\mathcal I$ of $S$ is
called {\it initial} (resp. {\it final}) if for every $\beta \in \mathcal I$ and $\alpha \in S$, if
$\alpha \leq \beta$ (resp. $  \alpha \geq \beta $) then $\alpha \in \mathcal I$.  \end{defn}

We review now some of the basic
notions of cuts of totally ordered sets. For further information
on cuts see, for example, [FKK] or [Weh].

\begin{defn} Let $T$ be a totally ordered set. 
A cut $\mathcal
A=(\mathcal A^{L}, \mathcal A^{R})$ of $T$ is a partition of $T$
into two subsets $\mathcal A^{L}$ and $\mathcal A^{R}$, such that,
for every $\alpha \in \mathcal A^{L}$ and $\beta \in \mathcal
A^{R}$, $\alpha<\beta$. We denote the set of all cuts of $T$ by $D(T)$ (``$D$" for Dedekind, who was the first
to consider a notion of cut). \end{defn}

Note that, viewing $T$ as a lattice, one can alternatively consider the ideal lattice of $T$, which is the set of all
ideals of $T$ (i.e., inital subsets of $T$). However, we prefer the notion of cuts, since we occasionally use both initial and final
subsets.

Let $T$ be a totally ordered set. The set of all cuts $\mathcal A=(\mathcal A^{L}, \mathcal A^{R})$
 of $T$ contains the two cuts $(\emptyset,T)$ and
 $(T,\emptyset)$; these are commonly denoted by $-\infty$ and
 $\infty$, respectively. Given $\alpha \in T$, we denote
$$(-\infty,\alpha]=\{\gamma \in T \mid \gamma \leq \alpha \}$$ and
$$(\alpha,\infty)=\{\gamma \in T \mid \gamma > \alpha \}.$$
One defines similarly the sets $(-\infty,\alpha)$ and
$[\alpha,\infty)$.

To define a cut one often writes $\mathcal A^{L}= \mathcal I$, meaning that
$\mathcal A$ is defined as $(\mathcal I, T \setminus \mathcal I)$ when $\mathcal I$ is an
initial subset of $T$. We define a (left) ordering on $D(T)$
by $\mathcal A \leq \mathcal B$ iff $\mathcal A^{L}
\subseteq \mathcal B^{L}$ (or equivalently $\mathcal A^{R}
\supseteq \mathcal B^{R}$). Given $S  \subseteq T$, $S^{+}$ is the smallest cut $\mathcal A$ such
that $S  \subseteq \mathcal A^{L}$. In particular, for $\alpha \in T$ we have
$\{\alpha\}^{+}=((-\infty,\alpha],(\alpha,\infty))$.

\begin{defn} A totally ordered set $T'$ is called a {\it Dedekind totally ordered set}
if there exists a totally ordered set $T$ such that $D(T) $ is isomorphic to $ T'$, as totally ordered sets. \end{defn}

\begin {rem} \label{def of phi} Note that there are two natural order preserving injections of totally ordered sets $\varphi_1,  \varphi_2:
T \rightarrow  D(T)$ defined in the following
way: for every $\alpha \in T$, $$\varphi_1 (\alpha) =
((-\infty,\alpha],(\alpha,\infty)) \text{ and } \varphi_2 (\alpha) =((-\infty,\alpha),[\alpha,\infty)). $$

\end{rem}



\begin{defn} Let $(S, \leq)$ be a poset and let $a,b \in S$. We write $a<b$ if $a \leq b$ and $a \neq b$; in this case, we say that $a$ is a {\it predecessor} of $b$.
We say that $a$ and $b$ are {\it immediate neighbors} in $S$ if
$a < b$ and there is no $c \in S$ such that $a < c < b$.
We also say that $a$ is an {\it immediate predecessor} of $b$ in $S$ or that $b$ is an {\it immediate successor} of $a$, if $a$ and $b$ are immediate neighbors in $S$.

\end{defn}

\begin{defn} Let $(S, \leq)$ be a poset. We say that $S$ satisfies (K1) if every nonempty totally ordered subset of $S$
has an infimum and a supremum. We say that $S$ satisfies (K2) if for all $a<b$ in $S$ there exist $a \leq c<d \leq b$ in $S$
such that $c$ and $d$ are immediate neighbors.

\end{defn}


\begin{defn} Let $(S, \leq)$ be a poset, let $ C \subseteq  S $ be a chain and let $a,b \in S$.
We say that $ C$ is a {\it maximal chain between $a$ and $b$} if $a$ is the smallest member of $C$,
$b$ is the greatest member of $C$, and
for any $x \in S \setminus C$ such that $a < x < b$,
one has $C \cup \{ x\}$ is not a chain (i.e., one cannot ``insert" an element of $S$ between the elements of $C$).
We say that $ C$ is a {\it maximal chain in $S$} if
for any $x \in S \setminus C$, one has $C \cup \{ x\}$ is not a chain.
\end{defn}

Let $(S, \leq)$ be a poset.
We define now two special subsets of $S$. Let
$$\text{IS}( S)=\{ a \in S \mid a \text{ has an immediate predecessor in }  S \};$$
and let $$\text{IP}( S)=\{ a \in S \mid a \text{ has an immediate successor in }  S \}.$$

\begin{rem} \label{phi1T=IP and phi2T=IP} Let $T$ be a totally ordered set and let $\varphi_1,  \varphi_2:
T \rightarrow  D(T)$ be as in Remark \ref{def of phi}. Then $\varphi_1(T)=\text{IS}(D(T))$ and $\varphi_2(T)=\text{IP}(D(T))$.

\end{rem}

Clearly, the class of all nonempty totally ordered sets strictly contains the class of all
Dedekind totally ordered sets; since, for example, a Dedekind totally ordered set satisfies (K1) and (K2).

\section{Immediate predecessors and immediate successors -- some general results}

In this section, we let $B$ be a nonempty set of sets. 
We equip $ B$ with the partial order of containment. 
We prove that if $B$ satisfies (K2), and is closed
under unions of nonempty chains (resp. intersections of nonempty chains),
then $\text{IS}(B)$ (resp. $\text{IP}(B)$) generate $B \setminus\{ \text{minimal elements of } B \}$ (resp.
$B \setminus\{ \text{maximal elements of } B \}$), in a sense to be clarified soon.
We also obtain a useful connection when $B$ is a chain.


The following lemma is a generalization of [Ka, Theorem 11]; the idea of the proof is quite similar.
We prove it here for the reader's convenience.

\begin{lem} \label{genKap} 
Let $X_1 \subset X_2 \in  B $. Then there exist
immediate neighbors in $ B $ between $X_1 $ and $X_2$ (i.e., there exist
$Y_1, Y_2 \in  B $ such that $Y_1 $ and $Y_2$ are immediate neighbors in $ B $
and $X_1 \subseteq Y_1 \subset Y_2 \subseteq X_2$) iff there exist $y \in X_2 \setminus X_1$ and a maximal chain
$ C \subseteq  B $ between $X_{1}$ and $X_{2}$, 
such that $\cup_{X \in C, y \notin X} X \in B$ and $\cap_{X \in C, y \in X} X \in B$.



\end{lem}

\begin{proof} $(\Leftarrow)$ Since $ C $ is a maximal chain
 between $X_{1}$ and $X_{2}$, $\cup_{X \in C, y \notin X} X $ and $\cap_{X \in C, y \in X} X $ are
immediate neighbors in $ B $ between $X_1 $ and $X_2$.
$(\Rightarrow)$ Let $Y_1 \subset Y_2$ be immediate neighbors in $ B $ between $X_1 $ and $X_2$.
Using Zorn's Lemma, let $C_{1} \subseteq B$ (resp. $C_{2} \subseteq B$) be a maximal chain between $X_{1}$ and $Y_{1}$
(resp. between $Y_{2}$ and $X_{2}$). Then, $y \in Y_{2} \setminus Y_{1}$ and $C=C_{1} \cup C_{2}$ are the required elements.

\end{proof}


By Lemma 2.1 we have,

\begin{lem} \label{kap1 implies kap2} If for every nonempty chain $ C \subseteq  B $, $\cup_{X \in  C} X \in  B$
and $\cap_{X \in  C} X \in  B$, then $B$ satisfies (K2).

\end{lem}



The following theorem is of utmost importance to our study.

\begin{thm} \label{union of initial subset of IS gives every nonminimal element}
Assume that $B$ satisfies (K2) and for every nonempty chain $ C \subseteq  \text{IS}(B) $, $\cup_{Z \in  C} Z \in  B$.
Let $X $ be an element of $B$
which is not minimal.
Then there exists a unique maximal nonempty initial subset $\mathcal I$ of $\text{IS}( B)$ (maximal with respect to containment) such that
$X=\cup_{Y \in \mathcal I}Y$. 
Furthermore, $X=\cup_{Y \in  E}Y$ for every maximal chain $E \subseteq  \mathcal I$.
\end{thm}

\begin{proof} Let $\mathcal I=\{ Y \in \text{IS}( B) \mid Y \subseteq X\}$.
Since $X$ is not minimal in $B$, there exists $Z_0 \in B$ such that  $Z_0 \subset X$.
Thus, since $B$ satisfies (K2), 
$\mathcal I \neq \emptyset$. Clearly, $\mathcal I$ is an initial subset of $\text{IS}( B)$.
Now, let $E $ be any maximal chain contained in $ \mathcal I$.
By the assumption on $ B$, $\cup_{Y \in E}Y \in  B$.
Assume to the contrary that
$\cup_{Y \in  E}Y \subset X$. Then, since $B$ satisfies (K2), 
there exists $\cup_{Y \in  E}Y  \subset Y_{1} \in \mathcal I$.
Hence, $E \cup \{ Y_{1}\} \subseteq \mathcal I$ is a chain strictly containing $E$, a contradiction.
Thus, $\cup_{Y \in  \mathcal I}Y = \cup_{Y \in  E}Y = X$.
Finally, it is clear that any subset $J$ of $\text{IS}( B)$ such that $X=\cup_{Y \in J}Y$
satisfies $J \subseteq \mathcal I$.


\end{proof}

It is not difficult to see that the assumptions in Theorem \ref{union of initial subset of IS gives every nonminimal element}
are sufficient but not necessary; 
however, if we omit any one of the assumptions then the theorem is not valid anymore. 
We note that the dual of Theorem \ref{union of initial subset of IS gives every nonminimal element} is also valid.
We state it here without a proof.

\begin{thm} \label{intersection of final subset of IP gives every nonmaximal element}
Assume that $B$ satisfies (K2) and for every nonempty chain $ C \subseteq  \text{IP}(B) $, $\cap_{Z \in  C} Z \in  B$. Let $X $ be an element of $B$
which is not maximal.
Then there exists a unique maximal nonempty final subset $\mathcal J$ of $\text{IP}( B)$ such that 
$X=\cap_{Y \in \mathcal J}Y$. 
Furthermore, $X=\cap_{Y \in  E}Y$ for every maximal chain $E \subseteq  \mathcal J$.
\end{thm}

\begin{rem} Note that the assumptions in the two previous theorems are weaker than the assumption that $B$ is closed under
unions and intersections of nonempty chains, in view of Lemma \ref{kap1 implies kap2}.
\end{rem}
The following lemma is obvious, but will be useful.

\begin{lem} \label{if B is a chain the the union of immediate successors does not reach X} Assume that $B$ is a chain. If $X \in \text{IS}(B)$ then the set $\cup_{Y \in \text{IS}(B), Y \subset X}Y$ is strictly contained in $X$.
If $X \in \text{IP}(B)$ then the set $\cap_{Y \in \text{IP}(B), Y \supset X}Y$ strictly contains $X$.

\end{lem}

When $B$ is a chain, we obtain the following useful connection between $B$ and the set of all cuts
of $\text{IS}(B)$ (and the set of all cuts of $\text{IP}(B)$).

\begin{thm} \label{when B is a chain}
Assume that $B$ satisfies (K2). Assume that either (i) $B$ has a smallest member, and
for every nonempty chain $ C \subseteq  \text{IS}(B) $, $\cup_{X \in  C} X \in  B$, or (ii) $B$ has a greatest member, and
for every nonempty chain $ C \subseteq  \text{IP}(B) $, $\cap_{X \in  C} X \in  B$.
If $B$ is a chain then $B$ is order isomorphic to $D(\text{IS}(B))$ and to $D(\text{IP}(B))$.
\end{thm}

\begin{proof} Assume (i) and define $\varphi: D(\text{IS}(B)) \rightarrow B$ by:
$$\varphi(\emptyset, \text{IS}(B))=X_0, $$ where $X_0$ is the smallest member of $B$; 
and for all $(\emptyset,\text{IS}(B)) \neq \mathcal A=(\mathcal A^L, \mathcal A^R) \in D(\text{IS}(B))$,
$$\varphi(\mathcal A)= \cup_{X \in \mathcal A^L} X .$$
By Theorem \ref{union of initial subset of IS gives every nonminimal element}, $\varphi$ is surjective.
By Lemma \ref{if B is a chain the the union of immediate successors does not reach X},
$\varphi$ is injective. Clearly, $\varphi$ is order preserving. Finally, to show that $B$ is order isomorphic
to $D(\text{IP}(B))$, use the natural order preserving bijection between $\text{IP}(B)$ and $\text{IS}(B)$.
Assuming (ii), the proof is very much alike and, of course, one can prove it
using Theorem \ref{intersection of final subset of IP gives every nonmaximal element}.


\end{proof}


\begin{rem} In view of Lemma \ref{kap1 implies kap2} and Theorem \ref{when B is a chain}, if $B$ is a chain
and is closed under unions and intersections of nonempty chains, then $B$ is
order isomorphic to $D(\text{IS}(B))$ and to $D(\text{IP}(B))$.
\end{rem}

\begin{rem} In Theorem \ref{when B is a chain}, if in (i) (resp. (ii)) $B$ does not have a smallest (resp. largest) element,
then $B$ is order isomorphic to $D(\text{IS}(B)) \setminus \{ (\emptyset, \text{IS}(B)) \}$ (resp. $D(\text{IP}(B)) \setminus \{ (\text{IP}(B),\emptyset) \}$).
\end{rem}

So, whenever $B$ satisfies (K2) and (i) or (ii) of Theorem \ref{when B is a chain}, $B$ is order isomorphic
to $D(\text{IS}(B))$ (and to $D(\text{IP}(B))$). It is natural to ask whether, whenever $B$ does not satisfy (K2), there
exists a subset $S \subseteq B$ such that $B$ is order isomorphic to $D(S)$. However, it is easy to see that such subset cannot exist
since a Dedekind totally ordered set satisfies (K2).

\section{A totally ordered set is order isomorphic to two special subsets of the prime spectrum of a valuation domain}

We start this section with an application of Theorem \ref{when B is a chain} to ring theory.


\begin{thm} \label{specOv is isomorphic to some dedekind set} Let $A$ be a ring (not necessarily commutative) such that $\text{Spec}(A)$ satisfies (K2).
If $\text{Spec}(A)$ is totally ordered
then
$\text{Spec}(A)$ is order isomorphic to $D(\text{IS}(\text{Spec}(A)))$ and to $D(\text{IP}(\text{Spec}(A)))$.
In particular, there exists a Dedekind totally ordered set which is
order isomorphic to $\text{Spec}(A)$.
\end{thm}

\begin{proof} An intersection of a nonempty chain of prime ideals of $A$ is a prime ideal of $A$,
and $A$ has a unique maximal ideal.
The assertion now follows from Theorem \ref{when B is a chain}.


\end{proof}



Recall from [BRV] that an ideal is called union-prime if it is a union of a chain of primes, but is
not prime. Following this definition, we define a ring to be union-prime free if it has no
union-prime ideals. Of course, any commutative ring is union-prime free and any ring
satisfying the ascending chain condition on prime ideals is union-prime free. 

It is still not known whether the prime spectrum of any ring satisfies the property (K2).
However, by Lemma \ref{kap1 implies kap2}, the prime spectrum of any union-prime free ring does satisfy (K2).
We will show in a subsequent paper that the class of rings whose prime spectra satisfy (K2) strictly contains the class
of union-prime free rings.

We note in passing that it follows from [Ba, Proposition 2.2 and Example 3.4] that the class of domains with totally ordered prime spectra strictly contains the class of valuation domains.
Of course, Theorem \ref{specOv is isomorphic to some dedekind set} applies to discrete valuation domains.
We deduce the converse in Corollary \ref{DT corresponds to a valuation domain},
after proving in Theorem \ref{T and IPSpec} one of the main results of this paper:
every totally ordered set is order isomorphic to the set of all prime ideals of $O_{v}$ having an immediate successor,
for some discrete valuation domain $O_{v}$.




It is well known that for every totally ordered abelian group $\Delta$ there exists a valuation domain whose value group is $\Delta$.
(See, for several constructions, [Bo, no. 3, Section 4, Example 6], [Ef, Example 4.2.4] and [Sc, Chapter 1, Section 6, Example 4]). 
Now, let $T$ be a totally ordered set.
Note that if $T=\emptyset$ then $D(T)=\{(\emptyset, \emptyset)\}$; in this case the prime spectrum of any field
is isomorphic to $D(T)$. So, we assume from now on that $T$ is not empty.
Let $G$ be a totally ordered group of rank one.
We consider the group $G^T$ with addition defined componentwise. Let $\Gamma$ denote the subgroup of $ G^T$ of
all $f \in G^T$ having a well ordered support, where the support of $f$ is $supp(f)=\{t \in T \mid f(t) \neq 0\}$.
We define a total ordering on $\Gamma$ (a left to right lexicographic order): for all $f,g \in \Gamma$, $f \leq g$ if $f=g$ or there exists
a smallest $t \in supp(f) \cup supp(g)$ such that $f(t)<g(t)$.
We denote by $O_{v}$ a valuation domain whose value group is $\Gamma$.

Recall that an isolated subgroup $H$ of $\Gamma$ is a subgroup of $\Gamma$ such that for every $0 \leq h \in H$
and for every $g \in \Gamma$, $0 \leq g \leq h$, one has $g \in H$.
It is well known that there exists
a one-to-one order reversing correspondence between the set of all prime ideals of a valuation
domain and the set of all isolated subgroups of the value group (see [End, p. 47]). Thus, we shall obtain several results
regarding isolated subgroups of $\Gamma$, and then use the correspondence mentioned above.

To my knowledge, the first to suggest the natural generalization of a discrete valuation domain was Dai in [Da].
He defined a discrete valuation domain of arbitrary rank
(whereas most of the authors discuss rank one discrete valuation domains, or finite rank) as a valuation domain such that for each pair of
prime ideals which are immediate neighbors (prime ideals $P_1 \supset P_2$ such that there is no prime ideal between them), their associated isolated subgroups $H_2 \supset H_1$, satisfy $H_{2}/H_{1} \approx \Bbb Z$. Equivalently, a discrete valuation domain is a valuation domain such that for each pair of
prime ideals $P_1 \supset P_2$ which are immediate neighbors, the ring $(O_{v}/P_{2})_{P_{1}/P_{2}}$ is a discrete rank one valuation domain.
As a side note, we mention that the term generalized discrete valuation ring was introduced in [Br]:
Brungs defined a generalized discrete valuation ring as a ring whose right ideals are well ordered by reverse inclusion.






\begin{rem} Note that, taking $G=\Bbb Z$ in our construction,
we obtain a discrete valuation domain, in view of the generalized definition.
\end{rem}


Denote by isolated$(\Gamma)$ the set of all isolated subgroups of $\Gamma$ ordered by inclusion.

We characterize now the isolated subgroups of $\Gamma$. We start by defining two types of isolated subgroups of $\Gamma$.
First, we define the {\it $T$-isolated subgroups}:
for all $t \in T$ let
$$H_{t}=\{ f \in \Gamma \mid f(s)=0 \ \text{ for all } s\in T \text{ such that } s<t \}. $$

Next, we define the {\it dual $T$-isolated subgroups}:
for all $t \in T$ let
$$dH_{t}=\{ f \in \Gamma \mid f(s)=0 \ \text{ for all } s\in T \text{ such that } s \leq t \}. $$

The following observation will be useful.

\begin{lem} \label{f in H ft neq 0 implies H contains Ht} Let $H$ be an isolated subgroup of $\Gamma$,
 let $f \in H$, and let $t \in T$. If $f(t) \neq 0$ then $H_t \subseteq H$. In particular, for all $j \geq t$, $H_j \subseteq H$.

\end{lem}

\begin{proof} Let $g \in H_t$ be a positive element; then $g(s)=0$ for every $s<t$. 
Now, let $k$ be the smallest element in the support of $f$. By the assumption on $f$, $k \leq t$. If $k<t$ then clearly
$g<f$ or $g<-f$ and thus $g \in H$. If $k=t$ then, since $G$ is of rank one,
there exists $z \in \Bbb Z$ such that $g(t) \leq zf(t) \in H$; thus $g \leq zf$ and $g \in H$.

\end{proof}

\begin{lem} \label{ISisolated=H_t} $\text{IS}(\text{isolated}(\Gamma))=\{ H_{t} \mid t \in T\} $; i.e., the set of all isolated subgroups of $\Gamma$
having an immediate predecessor is equal to the set of all $T$-isolated subgroups of $\Gamma$. 

\end{lem}

\begin{proof} $(\supseteq)$ Let $H=H_{t} $ be a $T$-isolated subgroup of $\Gamma$. If $t$ is the maximal element
of $T$ then by Lemma \ref{f in H ft neq 0 implies H contains Ht}, $H_{t} $ is the minimal nonzero isolated subgroup of $\Gamma$,
and thus $H_t$ is an immediate successor of the zero isolated subgroup. So, we may assume that $t$ is not the maximal element
of $T$; we prove that
$\cup_{k>t }H_{k}$ is an immediate predecessor of $H_{t}$. First, it is easy to see that $\cup_{k>t }H_{k} \subset H_t$;
indeed, there exists $f \in H_t$ such that $f(t) \neq 0$
and clearly $f \notin \cup_{k>t }H_{k}$. 
Next, we prove that for every isolated subgroup
$\cup_{k>t }H_{k} \subset H'$, one has $H_{t} \subseteq H'$. Let $H'$ be any isolated subgroup of $\Gamma$ strictly
containing $\cup_{k>t }H_{k}$ and let $g \in H' \setminus \cup_{k>t }H_{k}$; 
then there exists $j \in T$ such that  $j \leq t$ and $g(j) \neq 0$. 
Therefore, by Lemma \ref{f in H ft neq 0 implies H contains Ht}, $H_{t} \subseteq H'$. Hence,
$\cup_{k>t }H_{k}$ and $H_{t}$ are immediate neighbors in isolated$(\Gamma)$.

$(\subseteq)$ Let $H \in \text{IS}(\text{isolated}(\Gamma))$; then there exists an immediate predecessor $H'$ of $H$.
Let $f \in H \setminus H'$ and let $k \in T$ be the smallest element for which
$f(k) \neq 0$. Then $f \in H_k \setminus H'$ and thus $H' \subset H_{k}$; by Lemma
\ref{f in H ft neq 0 implies H contains Ht},  $ H_{k} \subseteq H$.
Therefore $H=H_{k}$.

\end{proof}

We present here, without a proof, the dual of Lemma \ref{ISisolated=H_t}: $$\text{IP}(\text{isolated}(\Gamma))=\{ dH_{t} \mid t \in T\}.$$

So, we have proved one of the main results of this paper:

\begin{thm} \label{T and IPSpec} There exists an order reversing bijection between $T$ and $\text{IS}(\text{isolated}(\Gamma))$,
defined by $t \rightarrow H_t$ for every $t \in T$.
Thus, there exists an isomorphism of totally ordered sets between $T$ and $\text{IP}( {\text Spec}(O_{v}) )$
(and hence between $T$ and $\text{IS}( {\text Spec}(O_{v}) ))$.


\end{thm}

Now, we can prove that any nonzero isolated subgroup of $\Gamma$ is of the form $\cup_{t \in \mathcal A^R}H_{t}$
for some cut $\mathcal A=(\mathcal A^L, \mathcal A^R)$ of $T$.

\begin{prop} \label{H=cupH_t} Let $H \neq \{ 0 \}$ be an isolated subgroup of $\Gamma$. Then there exists a final subset $\mathcal J$ of $T$
such that $H=\cup_{t \in \mathcal J}H_{t}$.

\end{prop}

\begin{proof} 
By Lemma \ref{ISisolated=H_t}, $\text{IS}(\text{isolated}(\Gamma))=\{ H_{t} \mid t \in T\} $.
By Theorem \ref{union of initial subset of IS gives every nonminimal element}, there exists a nonempty initial subset $\mathcal I$ of
$\{ H_{t} \mid t \in T\}$ such that $H=\cup_{H_t \in \mathcal I}H_{t}$. Let $\mathcal J=\{ t\in T \mid H_t \in \mathcal I\}$;
then $\mathcal J$ is a final subset of $T$ and $H=\cup_{t \in \mathcal J}H_{t}$.


\end{proof}

\begin{cor} \label{DT corresponds to a valuation domain} Let $T$ be a totally ordered set. Then there exists a discrete valuation domain $O_{v}$ such that $\text{Spec}(O_{v})$ is isomorphic to $D(T)$ (as totally ordered sets).



\end{cor}

\begin{proof} Let $\Gamma$ be the totally ordered group defined earlier and let $O_{v}$ be a discrete valuation domain
whose value group is $\Gamma$.
 Let $\psi:\text{isolated}(\Gamma) \rightarrow \text{Spec}(O_{v})$ be the well-known order reversing bijection given by
$H \rightarrow \{ a \in O_{v} \mid v(a) \notin H\}$.
Let $\varphi:D(T) \rightarrow \text{isolated}(\Gamma)$ be the function defined by
$\varphi(\mathcal A)= \cup_{t \in \mathcal A^R} H_t$ for all $(T,\emptyset) \neq \mathcal A=(\mathcal A^L, \mathcal A^R) \in D(T)$
and $\varphi(T, \emptyset)=\{ 0\}$. By Proposition \ref{H=cupH_t}, $\varphi$ is surjective. By Lemma \ref{ISisolated=H_t} and Lemma
\ref{if B is a chain the the union of immediate successors does not reach X}, $\varphi$ is injective.
It is easy to see that $\varphi$ is order reversing.
Therefore, $\mu= \psi \varphi : D(T) \rightarrow \text{Spec}(O_{v})$ is an order preserving bijection.

\end{proof}

We note that Corollary \ref{DT corresponds to a valuation domain} (without the ``discrete" part) can be deduced from [Le, Corollary 3.6]
(which says that a totally ordered set $T$ is order isomorphic to the prime spectrum of some commutative valuation ring
iff $T$ satisfies (K1) and (K2)).

\begin{rem} Note that, viewing $T$ as a subset of $D(T)$ under the natural injection $\varphi_2$ presented in Remark \ref{def of phi},
 $\mu \mid_T : T \rightarrow \text{IP}({Spec}(O_{v}))$ is the isomorphism mentioned in Theorem \ref{T and IPSpec}.
\end{rem}

The symbol $\cong$ in the next corollary means order isomorphism.

\begin{cor} \label{equivalent conditions} Let $T$ be a totally ordered set. The following conditions are equivalent:

(a) $T \cong Spec(A_1)$ for some discrete valuation domain $A_1$.

(b) $T \cong Spec(A_2)$ for some commutative valuation ring $A_2$.

(c) $T \cong Spec(A_3)$ for some ring $A_3$ whose prime spectrum satisfies (K2) and is totally ordered.

(d) $T$ is a Dedekind totally ordered set.

(e) $T$ satisfies (K1) and (K2).

\end{cor}

\begin{proof} The implications $(a)\Rightarrow(b)\Rightarrow(c)$ are trivial. $(c)\Rightarrow(d)$ is by Theorem
\ref{specOv is isomorphic to some dedekind set}, and $(d)\Rightarrow(a)$ is by Corollary
\ref{DT corresponds to a valuation domain}. $(b) \Leftrightarrow (e)$ is by [Le, Corollary 3.6].

\end{proof}

In particular, from the point of view of poset theory, there is no difference between the prime spectrum of a
discrete valuation domain and the prime spectrum of a noncommutative ring
whose prime spectrum satisfies (K2) and is totally ordered.







Let us give a slightly different perspective on the correspondence $\psi$ mentioned in the proof of
Corollary \ref{DT corresponds to a valuation domain}. We define the {\it $T$-prime ideals} of $O_{v}$: for every $t \in T$ let
$$P_{t}=\{ a \in O_{v} \mid \text{there exists } k \leq t \text{ such that } (v(a))(k) \neq 0 \}.$$
It is straightforward to verify that for each $t \in T$, $P_{t}$ is indeed a prime ideal of $O_{v}$.
Moreover, it is not difficult to see that for each non-maximal $t \in T$, one has $\psi(\cup_{j>t}H_j)=P_{t}$;
if $T$ has a maximal element, say $t_0$, then $\psi(\{0\})$, which is the maximal ideal of $O_{v}$, equals $P_{t_0}$.
In this way, we obtain all $T$-prime ideals of $O_{v}$.
More generally, we have the following:

\begin{rem} In light of Proposition \ref{H=cupH_t}, the well-known order reversing correspondence $\psi:\text{isolated}(\Gamma) \rightarrow \text{Spec}(O_{v})$
can be defined by: for every $(T,\emptyset) \neq \mathcal A=(\mathcal A^L, \mathcal A^R) \in D(T)$,
$\psi(\cup_{j \in \mathcal A^R}H_j)= \cup_{t \in \mathcal A^L} P_{t}$,
and $\psi(\{0\})=\cup_{t \in T} P_{t}$.
So, we have a new way of viewing the correspondence between $\text{isolated}(\Gamma)$ and $\text{Spec}(O_{v})$,
in terms of cuts of $T$.
\end{rem}

\begin{rem} Note that taking $G \cong \Bbb Z$ in our construction, the $T$-prime ideals of $O_{v}$ are precisly
the principal prime ideals of $O_{v}$.
\end{rem}


It is interesting to point out another known connection between the notion of cuts and that of a valuation.
In [Sa1], given a totally ordered abelian group, we constructed a cut monoid. We showed that
for any algebra over a valuation domain, there exists a function, called the filter quasi-valuation, which is induced by the algebra and the valuation; the values of the filter quasi-valuation lie inside the cut monoid.
It turned out that one can prove some interesting properties regarding
algebras over valuation domains using these quasi-valuations.
It would be interesting to know if one can generalize the results of this paper to
algebras over valuation domains and certain cuts of partially ordered sets.
For more information on quasi-valuations see [Sa1], [Sa2] and [Sa3].

We close this paper with an algebraic geometry point of view.
Let $A$ be a ring (not necessarily commutative) whose prime spectrum satisfies (K2) and is totally ordered.
We endow Spec($A$) with the usual
Zarisky topology (even when $A$ is not commutative; see [OV, p. 36]).
It is clear that a subset $\mathcal K$ of Spec$A$ is closed
in the Zarisky topology iff $\mathcal K=\emptyset$ or $\mathcal K$ is a final subset of the form $[P, \infty)$, for some $P \in \text{Spec}(A)$
(note that $[P, \infty)$ is just the set of prime ideals of $A$ containing or equal to $P$).
Viewing $\text{IP}(\text {Spec}(A))$ as a subspace of Spec$(A)$, we obtain the following:

\begin{prop} \label{closed=final} The closed subsets of $\text{IP}( \text{Spec}(A))$ are precisely the final subsets of $\text{IP}( \text{Spec}(A))$.

\end{prop}

\begin{proof} It is clear that a closed subset of $\text{IP}( \text{Spec}(A))$ is a final subset of $\text{IP}( \text{Spec}(A))$. We prove the converse.
Let $\mathcal J$ be a final subset of $\text{IP}( \text{Spec}(A))$. If $\mathcal J=\emptyset$ then
$\emptyset \cap \text{IP}( \text{Spec}(A))=\emptyset$ (or $\{ M \} \cap \text{IP}( \text{Spec}(A))=\emptyset$, where $M$ is the maximal ideal of $A$).
If $\mathcal J \neq \emptyset$ then $$[\cap_{P \in \mathcal J}P, \infty) \cap \text{IP}( \text{Spec}(A))=\mathcal J;$$
indeed, $\supseteq$ is obvious and $\subseteq$ is valid by Lemma \ref{if B is a chain the the union of immediate successors does not reach X}.
\end{proof}

\begin{rem} Note that the previous proposition is valid without the assumption that Spec$(A)$ satisfies (K2). \end{rem}

\begin{cor} \label{1:1 correspondence between the points of Spec and the closed subsets} There is a 1:1 correspondence between the points of Spec$(A)$ and the closed subsets of $\text{IP}( \text{Spec}(A))$.
\end{cor}

\begin{proof} By Theorem \ref{specOv is isomorphic to some dedekind set} and Proposition \ref{closed=final}.

\end{proof}

It is not difficult to see that one can obtain similar results (as in Proposition \ref{closed=final} and Corollary 
\ref{1:1 correspondence between the points of Spec and the closed subsets}) for $\text{IS}( \text{Spec}(A))$.

Assume that Spec$(A)$ has more than one point.
Considering $\text{IP}( \text{Spec}(A))$ as a subset of Spec$(A)$, it is not difficult to see, using the property (K2) (or Theorem
\ref{intersection of final subset of IP gives every nonmaximal element}), that it is dense in Spec$(A)$.
For $\text{IS}( \text{Spec}(A))$, the situation is a bit different. In fact, $\text{IS}( \text{Spec}(A))$ is dense in Spec$(A)$
iff the minimal prime ideal of $A$ is not an immediate predecessor; the proof is not difficult and is left for the reader.



Now, let $T$ be a totally ordered set. In view of Corollary \ref{DT corresponds to a valuation domain},
$D(T)$ has an induced Zarisky topology defined on it:
the topology in which the closed subsets are $\emptyset$ and the subsets of the form $[C,\infty)$ for some $C \in D(T)$.
In particular, the induced Zarisky topology on a Dedekind totally ordered set is compact, $T_0$ and sober (recall that a topological space is called sober if every nonempty irreducible closed subset has a unique generic point).
Also recall that the cop-topology on a poset $(S,\leq)$ is defined as the topology which has $\{ x \in S \mid x \geq s\}_{s \in S}$
as a subbasis for the closed sets.
It is now clear that the induced Zarisky topology on a Dedekind totally ordered set is equal to its cop-topology.



Finally, in view of Remarks \ref{def of phi} and \ref{phi1T=IP and phi2T=IP} and the above discussion, we have:

\begin{rem}
If $T \neq \emptyset$ then $\varphi_2(T)$ is dense in $D(T)$; and $\varphi_1(T)$ is dense in $D(T)$ iff $(\emptyset,T)$, the minimal element of $D(T)$,
is not an immediate predecessor, iff $T$ has no minimal element.
Thus, in view of Theorem \ref{T and IPSpec} and Proposition \ref{closed=final}, $T$ has also an induced Zarisky topology defined on it:
the topology in which the closed subsets are the final subsets of $T$.

\end{rem}

Department of Mathematics, Sce College, Ashdod 77245, Israel.

{\it E-mail address: sarusss1@gmail.com}

\end{document}